\documentclass[11pt,twoside, reqno]{amsart}

\usepackage{amsmath}
\usepackage{amsthm}
\usepackage{amsfonts, amssymb}
\usepackage{mathrsfs}
\usepackage[all]{xy}
\usepackage{url}

\usepackage{latexsym}
\usepackage[dvips]{graphicx}

\theoremstyle{plain}
\newtheorem{lema}{Lemma}
\newtheorem{prop}[lema]{Proposition}
\newtheorem{teo}[lema]{Theorem}

\newtheorem*{intro1}{Theorem \ref{main}}
\newtheorem*{intro2}{Theorem \ref{producto1}}
\newtheorem{coro}[lema]{Corollary}
\theoremstyle{remark}

\newtheorem{obs}[lema]{Remark}
\theoremstyle{definition}

\newtheorem{ej}[lema]{Example}

\newcommand{\Z}{\mathbb{Z}}

\newcommand{\R}{\mathbb{R}}

\newcommand{\g}{\textrm{Cl}}

\newcommand{\aut}{\textrm{Aut}}

\renewcommand{\g}{\mathfrak{g}}
	
\begin{document}

\title[Regular and semi-regular representations of groups by posets]{Regular and semi-regular representations of groups by posets}

\author[J.A. Barmak]{Jonathan Ariel Barmak $^{\dagger}$}

\thanks{$^{\dagger}$ Researcher of CONICET. Partially supported by grant PICT 2019-2338, PICT-2017-2806, PIP 11220170100357CO, UBACyT 20020190100099BA, UBACyT 20020160100081BA}

\address{Universidad de Buenos Aires. Facultad de Ciencias Exactas y Naturales. Departamento de Matem\'atica. Buenos Aires, Argentina.}

\address{CONICET-Universidad de Buenos Aires. Instituto de Investigaciones Matem\'aticas Luis A. Santal\'o (IMAS). Buenos Aires, Argentina. }

\email{jbarmak@dm.uba.ar}

\begin{abstract}
By a result of Babai, with finitely many exceptions, every group $G$ admits a semi-regular poset representation with three orbits, that is, a poset $P$ with automorphism group $\aut(P)\simeq G$ such that the action of $\aut(P)$ on the underlying set is free and with three orbits. Among finite groups, only the trivial group and $\Z_2$ have a regular poset representation (i.e. semi-regular with one orbit), however many infinite groups admit such a representation. In this paper we study non-necessarily finite groups which have a regular representation or a semi-regular representation with two orbits. We prove that if $G$ admits a Cayley graph which is locally the Cayley graph of a free group, then it has a semi-regular representation of height 1 with two orbits. In this case we will see that any extension of the integers by $G$ admits a regular representation. Applications are given to finite simple groups, hyperbolic groups, random groups and indicable groups. 
\end{abstract}



\maketitle


\section{Introduction} \label{content}

A \textit{representation} of a group $G$ by a poset is a poset $P$ whose automorphism group $\aut (P)$ is isomorphic to $G$, together with an isomorphism $G\to \aut (P)$. In other words, it is a faithful action of $G$ on $P$ such that every automorphism of $P$ is induced by the action. A representation of $G$ by $P$ is called \textit{semi-regular} if the action of $G$ on the underlying set of $P$ is semi-regular (free), that is the stabilizer of each point $x\in P$ is trivial. In this case, the orbit of each point is in bijection with $G$. A representation is \textit{regular} if it is semi-regular with a unique orbit (i.e. transitive).

Birkhoff proved that every finite group $G$ has a semi-regular poset representation with $|G|+1$ orbits \cite{Bir}, and Frucht then proved that $d+2$ orbits suffice, where $d$ is the cardinality of any generator set of $G$ \cite{Fru}. Babai proved in \cite[Corollary 0.14]{Bab2}, \cite[Corollary 4.3]{Bab}, the surprising result that with finitely many exceptions every group (non-necessarily finite) admits a semi-regular poset representation with three orbits. Babai's proof uses his characterization of groups admitting a \textit{digraphical regular representation} (defined below), which in turn analyzes properties of generator sets of different classes of groups. In \cite{Bar} we gave a short self-contained proof that every finitely generated group admits a semi-regular poset representation with four orbits. 


Note that if a group $G$ admits a regular representation by a poset $P$, then $P$ is homogeneous (i.e. given two points in $P$, there is an automorphism of $P$ mapping one to the other). If $G$ is finite, this means that $P$ is discrete (different points are not comparable), so $G$ is a symmetric group. Since the action is free, $G$ must be trivial or $\Z_2$. The classification of infinite groups admitting a regular poset representation, however, is a non-trivial problem. Regular representations of groups by totally ordered sets have been studied in \cite{Ohk,GGHS}. If a group admits one such representation, it is isomorphic to a subgroup of $\R$.

\medskip

\noindent{\textbf{Problem.} Classify those (finite, infinite) groups $G$ which admit a semi-regular poset representation with exactly two orbits and those (infinite) groups that admit a regular poset representation. }

\medskip

Similar classification problems for finite or infinite groups have been addressed and solved in the context of graphs. Finitely generated groups admitting a \textit{graphical regular representation} have been characterized after a long series of articles of different authors \cite{Godsil, Het, Imr1, Imr2, Imr3, Mca, NW, NW2, Sab}. A graphical regular representation of a group $G$ is a graph $\Gamma$ with $\aut (\Gamma)\simeq G$ and such that $\aut(\Gamma)$ acts regularly on the vertex set. On the other hand the proof of the main result in \cite{Bab3} shows that with finitely many exceptions every finitely generated group admits a \textit{graphical semi-regular representation} with two orbits. Other classification problems studied include digraphical regular representations \cite{Bab2,Bab}, graphical and digraphical semi-regular representations with more orbits \cite{DFS}, regular and semi-regular representations by tournaments with two orbits \cite{BI,Godsil2}, oriented regular representations \cite{MS}.

Among semi-regular poset representations with two orbits, there are some which are perhaps easier to describe: those which are given by posets of height 1 (i.e. the longest chain has two elements). Because of their relevance in this work and their resemblance to Cayley graphs, they will receive the name of \textit{Cayley (poset) representations}. For finite groups every semi-regular representation with two orbits is of this type.

Recall that the girth $\g(\Gamma)$ of a graph $\Gamma$ is the length of its shortest cycle (infinite if the graph is a forest). When $S=\{s_x\}_{x\in X}$ is a generator set of a group $G$, the girth $\g(\Gamma(G,S))$ of the Cayley graph can be interpreted as the shortest non-trivial word in the free group $F(X)$ generated by $X$ which is trivial in $G$ when evaluated in the $s_x$. We will prove the following

\begin{intro1}
Let $G$ be a group generated by two elements $x,y$. Suppose that for every non-trivial word $w$ in the free group of rank 2, with length smaller than or equal to $21$, $w(x,y)\neq e\in G$, that is $\g(\Gamma(G,\{x,y\}))>21$. Then $G$ admits a Cayley representation. 
\end{intro1}

Using ideas from Weigel, Dixon, Pyber, Seress and Shalev we will prove that, with finitely many exceptions, every finite simple group admits a Cayley representation (Corollary \ref{corosimple}). A second application of Theorem \ref{main} concerns random groups (in Gromov's density model, and in the Arzhantseva and Ol'shanskii few relators model). Using results by Ollivier, and by Arzhantseva and Ol'shanskii regarding Dehn presentations and sixth groups we will show that for two generators a random group has a Cayley representation (Corollaries \ref{corodensity} and \ref{corofew}). In the density model we require $d<\frac{1}{5}$.  

Cayley representations turn out to be important when studying regular representations. 

\begin{intro2}
Let $G\neq \Z_2$ be a group which admits a Cayley representation. Then every extension of the integers by $G$ has a regular poset representation.
\end{intro2} 

With similar ideas to those used in the proof of Theorem \ref{producto1} we will prove that with finitely many exceptions every indicable group admits a semi-regular representation with two orbits, though not neccesarily a Cayley representation.

\section{Cayley, Haar and digraphical regular representations} \label{sectioncayley}

We begin with the study of semi-regular representations with two orbits. Note first that for a group $G$, the fact that there exists a poset $P$ with $\aut(P)\simeq G$ and the action of $\aut(P)$ on $P$ having two orbits, does not imply that $G$ admits a semi-regular poset representation with two orbits. For example $\Z_2^2$ is the automorphism group of the poset with underlying set $\{0,1,0',1'\}$ where $0,1<0',1'$, but $\Z_2^2$ does not admit a semi-regular representation with two orbits.

We turn to the special case of representations of height 1. Let $G$ be a group and let $S\subseteq G$ be a non-empty subset. We define the \textit{Cayley poset} $P(G,S)$ as follows. The underlying set is a union of two copies $G$, $G'=\{g'|\ g\in G\}$ of $G$. If $g\in G$ and $h\in gS$, then $g<h'$. No other pair of different points are comparable. The regular action of $G$ on each copy of $G$, by left multiplication, gives an action $L:G\to \aut(P(G,S))$ of $G$ on $P(G,S)$. If there are no other automorphisms in $P$ than those induced by (in the image of) $L$, $P(G,S)$ is a semi-regular representation of $G$ with two orbits. In this case $P(G,S)$, or more precisely $L:G\to \aut(P(G,S))$, is called a \textit{Cayley representation} of $G$. 

Note that if a \textsl{finite} group is represented semi-regularly with two orbits by a poset $P$, then $P$ has height equal to 1, as points in the same orbit cannot be comparable ($x<gx$ would lead to an infinite chain). Let $G$ be a non-necessarily finite group and let $P$ be a semi-regular representation with two orbits and of height 1. Then the orbits are the set of minimal points and the set of maximal points. Thus, the set of minimal points can be identified with $G$, the set of maximal points with $G'$, and the action of $G$ on $P$ can be assumed to be the left regular action on each copy of $G$. If we denote by $S$ the set of elements $h\in G$ such that $e<h'$, then $S\neq \emptyset$ and $P=P(G,S)$. Indeed, given $g,h\in G$, we have that $g<h'$ if and only if $e<(g^{-1}h)'$, which is equivalent to $g^{-1}h$ being in $S$, i.e. $h\in gS$. 

In conclusion the semi-regular representations with two orbits and of heght 1 of a group $G$ are the Cayley representations of $G$, and for $G$ finite, every semi-regular representation with two orbits is of this form. Our problem, in the finite case, is then to classify those finite groups $G$ for which there is a subset $S$ with $\aut(P(G,S))$ being equal to the automorphisms induced by $L$. An infinite group $G$, however, could admit a semi-regular representation with two orbits of a different kind. Let $P$ be constructed from two copies $\Z, \Z'$ of $\Z$ in which $i,i'$ are smaller than $j$ every time $i<j$, where $<$ is the usual order of the integers. This is a semi-regular representation of $\Z$ with two orbits. The difference between Cayley representations and general semi-regular representations with two orbits, and the relevance of the first when studying regular representations will be clear in Section \ref{sectionregular}.

Let $G$ be a non-necessarily finite group and $\emptyset \neq S\subseteq G$. Then $P(G,S)$ is a Cayley representation of $G$ if and only if every automorphism fixing the minimal point $e\in G\subseteq P(G,S)$ is the identity. Indeed if $\varphi \in \aut(P(G,S))$, then $\varphi(e)$ must be a minimal point, say $g\in G\in P(G,S)$. Let $L(g)$ be the corresponding automorphism induced by $L:G\to P(G,S)$. Thus, $L(g)^{-1}\varphi$ fixes $e$, and if this implies that $L(g)^{-1}\varphi=1_{P(G,S)}$, then $\varphi=L(g)$.  

The following result is easy to prove and details are left to the reader.

\begin{prop} \label{pnueva}
Let $G$ be a group, $\emptyset \neq S \subsetneq G$. Then

\noindent (i) For every $h\in G$, $P(G,S)$ is a Cayley representation of $G$ if and only if $P(G,Sh)$ is.

\noindent (ii) For every automorphism $\psi$ of $G$, $P(G,S)$ is a Cayley representation of $G$ if and only if $P(G,\psi (S))$ is.

\noindent (iii) $P(G,S)$ is a Cayley representation of $G$ if and only if the \textit{complement} $P(G,G\smallsetminus S)$ is.
\end{prop}

\begin{obs} \label{generan}
If $P$ is a semi-regular poset representation of a group $G$ different from the trivial group $1$ and $\Z_2$, then $P$ is connected. Indeed if $P$ is not connected, by semi-regularity each component must have trivial automorphism group. If the components are pairwise non-isomorphic, then $\aut(P)$ is trivial, a contradiction. Let $C,C'$ be two isomorphic components of $P$. If there is no other component, then $\aut(P)=\Z_2$, a contradiction. If there is another component, there is a non-trivial automorphism of $P$ fixing every point in that component, contradicting semi-regularity again. 

In the case of Cayley posets $P(G,S)$, the fact that $P(G,S)$ is connected is equivalent to $SS^{-1}=\{st^{-1}| \ s,t\in S\}$ being a generating set of $G$, which in particular implies that $S$ generates $G$.
\end{obs}

A \textit{Haar graphical representation} of a group $G$ is a semi-regular representation by a bipartite graph with the orbits being the parts of the bipartition. In other words, it is a bipartite graph $B$ with parts $V_1,V_2$ together with an isomorphism between $G$ and $\aut(B)$, such that the action of $G$ on the vertex set of $B$ is semi-regular with the two orbits being $V_1$ and $V_2$. Every representation of this kind is isomorphic to a graph $B(G,S)$ for some $S\subseteq G$. This is the bipartite graph with parts $G,G'$ and edges $(g,(gs)')$ for $g\in G, s\in S$. The characterization of graphs admitting a Haar graphical representation is an open problem \cite{DFS}.

The comparability graph $C(P(G,S))$ of $P(G,S)$ (in this case this is the underlying undirected graph of the Hasse diagram) is a bipartite graph, and every automorphism of the poset induces an automorphism of the graph. Thus, if $G$ admits a Haar graphical representation, it admits a Cayley representation. The converse is not true as no abelian group admits a Haar graphical representation \cite{DFS}, but they can have a Cayley representation (see Example \ref{ciclico} below).

\bigskip

Recall that a digraphical regular representation of a group $G$ is a digraph $\Gamma$ such that $\aut(\Gamma)\simeq G$ and the action of $\aut(\Gamma)$ on the vertex set of $\Gamma$ is regular. Babai proved the following characterization \cite{Bab2, Bab}.

\begin{teo} (Babai) \label{DRR}
A non-necessarily finite group $G$ admits a digraphical regular representation if and only if $G$ is different from $\Z_2^2, \Z_2^3, \Z_2^4, \Z_3^2$ and the quaternion group $Q_8$.
\end{teo}

Babai observes that every digraphical regular representation induces a semi-regular poset representation with three orbits \cite[Proposition 7.3]{Bab2}. This representation is constructed by taking three copies $G,G',G''$ of $G$ and setting $g<g'<g''$ for every $g\in G$ and $g<h''$ if $(g,h)$ is an edge in the digraphical regular representation. For $Q_8$ Babai provides a separate construction of a semi-regular poset representation with three orbits. Thus, every group different from $\Z_2^2, \Z_2^3, \Z_2^4, \Z_3^2$ admits a semi-regular poset representation with three orbits. The converse of this statement is not analyzed in \cite{Bab2, Bab}.

\begin{prop} \label{zeta22}
The group $\Z_2^2$ does not admit a semi-regular poset representation with three orbits.
\end{prop}
 
The proof we give essentially analyzes all possible candidates of representations, and it is included at the end of this article. We do not inspect here the cases $\Z_2^3, \Z_2^4$ and $\Z_3^2$.

The idea for proving that the five groups in Theorem \ref{DRR} have no digraphical regular representation is to show that in each of those cases for every subset $S\subseteq G$ there exists an automorphism $\varphi\neq 1_G$ of $G$ such that $\varphi (S)=S$. Instead of applying this idea to our problem directly, we observe the following.

\begin{obs} \label{DRRremark}
If a group $G$ admits a Cayley representation $P(G,S)$, then it admits a digraphical regular representation. Define the digraph $\Gamma$ with vertex set $G$ and edges $(g,gs)$ for every $g\in G, s\in S$. Let $\varphi \in \aut(\Gamma)$. Define $\overline{\varphi}:P(G,S)\to P(G,S)$ by $\overline{\varphi}(g)=\varphi(g)$ and $\overline{\varphi}(g')=\varphi(g)'$ for every $g\in G$. It is easy to see that $\overline{\varphi}$ is an automorphism of $P(G,S)$, so by assumption it coincides with $L(g)$ for some $g\in G$. This implies that $\varphi$ corresponds to the left multiplication by $g$. Thus, $\Gamma$ is a digraphical regular representation of $G$. 
\end{obs}

The converse of this statement does not hold. Moreover, we have the following result.

\begin{prop} \label{contraejemplos}
The groups $G=\Z_3,\Z_4,\Z_5,\Z_6,\Z_7,\Z_2^2,\Z_2^3,\Z_2^4,\Z_3^2,S_3$ and $Q_8$ do not admit a semi-regular poset representation with two orbits.
\end{prop}
\begin{proof}
The cases $\Z_3,\Z_4,\Z_5$ and $\Z_7$ are covered by Corollary 12 of \cite{BB}. For $\Z_2^2,\Z_2^3,\Z_2^4,\Z_3^2, Q_8$ the result follows from Theorem \ref{DRR} and Remark \ref{DRRremark}. Suppose then that $G=\Z_6=\{0,1,2,3,4,$ $5\}$. Assume there exists $\emptyset \neq S \subseteq \Z_6$ such that $\aut(P(G,S))$ is a Cayley representation of $\Z_6$. Then $S\neq \Z_6$ and by Proposition \ref{pnueva}, we may assume that $|S|\le 3$ and that $0\in S$. By Remark \ref{generan}, $P(G,S)$ is connected. If $|S|=2$, $P(G,S)$ is a crown, and there is an involution fixing $0$. Therefore we may assume $|S|=3$, and by Proposition \ref{pnueva} there are only two representatives to analyze: $S=\{5,0,1\}$, $\{0,1,3\}$. In the first case the involution $k\to -k$, $k'\to (-k)'$ fixes $0$. In the second case there is an involution which transposes $2$ with $5$ and $0'$ with $3'$, and fixes any other point. 

Finally, suppose $\emptyset \neq S \subsetneq S_3$ is such that $P(S_3,S)$ is a Cayley representation of $S_3=\{e,(12),(13),(23),(123),(132)\}$. Again by Proposition \ref{pnueva} we may assume that $|S|\le 3$ and that $e\in S$. Since $S_3$ is not cyclic, by Remark \ref{generan}, $|S|=3$. By applying an automorphism of $S_3$ there are only two cases to analyze $S=\{e,(12),(13)\}, \{e,(12),(123)\}$. In the first case there is an involution which transposes $(12)$ with $(13)$, $(123)$ with $(132)$, $(12)'$ with $(13)'$ and $(123)'$ with $(132)'$. In the second, there is an involution transposing $(12)$ with $(23)$, $(123)$ with $(132)$, $e'$ with $(123)'$ and $(13)'$ with $(23)'$. 
\end{proof}

\section{Girth of Cayley graphs and representations with two orbits}

Let $G$ be a group and let $\emptyset \neq S\subseteq G$ be a subset. For $g\in G$ we define the \textit{$S$-neighborhood} of $g$ by $N_S(g)=gSS^{-1}=\{gst^{-1} | \ s,t\in S\} \subseteq G$. Note that two minimal elements $g,h\in G$ of the Cayley poset $P=P(G,S)$ have a common upper bound if and only if $h\in N_S(g)$. Thus, if $\varphi \in \aut(P)$, $\varphi (N_S(g))=N_S(\varphi(g))$ for every $g\in G$. Given $g,h\in G$, their \textit{affinity} is $\alpha(g,h)=\# (N_S(g)\cap N_S(h))$ (this could be infinite if $S$ is). Then $\alpha(\varphi(g),\varphi(h))=\alpha(g,h)$ for every automorphism $\varphi$.

\begin{ej} \label{ciclico}
We claim that if $n\ge 9$ and $S=\{0,1,3\}\subseteq \Z_n$, then $P=P(\Z_n,S)$ is a Cayley representation of $\Z_n$. Let $\varphi \in \aut(P)$ be an automorphim which fixes $0\in P$. We only need to prove that $\varphi$ is the identity $1_P$.  

It is easy to see that for $i\in \Z_n$, $N_S(i)=\{i-3,i-2,i-1,i,i+1,i+2,i+3\}$. Since $n\ge 7$, $\#N_S(i)=7$ for every $i$. Since $n\ge 8$, $\alpha(i,i+1)=6$ for every $i$. Since $n\ge 9$, $\alpha(i,j)=6$ only when $j\in \{i-1,i+1\}$. Since $\varphi$ preserves affinity, and $0$ is fixed, then $\varphi(1)\in \{-1,1\}$. However, if $\varphi(1)=-1$, by induction $\varphi(i)=-i$ for every $i$. This is a contradiction as $\{0,2,3\}$ has an upper bound, while $\{0,-2,-3\}$ does not. Thus $\varphi(1)=1$ and by induction $\varphi(i)=i$ for every $i$. Each maximal element $i'$ is uniquely determined by its smaller elements $i,i-1,i-3$. Thus $\varphi(i')=i'$ for every $i$.
\end{ej}

Note that $9$ is three times the diameter of $S=\{0,1,3\}$ in the Cayley graph $\Gamma(\Z_n, \{1\})$. Also note that the argument in Example \ref{ciclico} can be used with no changes to prove that $\Z$ admits a Cayley representation.

In \cite{BB} it is proved that $\Z_8$ also admits a Cayley representation. A similar argument is used with $S=\{0,1,2,4\}$ and a different version of the notion of $S$-neigborhood. In our next result both choices of $S$ will be used simultaneously to represent groups generated by two elements.

Suppose $\{x,y\}$ is a set of generators of a group $G$, and let $\Gamma=\Gamma (G,\{x,y\})$ be the corresponding Cayley graph, i.e. the undirected graph with vertex set $G$ and an edge with vertices $g,gx$ and another with vertices $g,gy$ for each $g\in G$ (in particular $\Gamma$ has loops if any generator is the identity and it has multiple edges if one generator has order two or if it  equals the other or its inverse). Once again, recall that the girth of a graph is the smallest length of a cycle (infinite if the graph is a forest), so the girth $\g(\Gamma)$ of $\Gamma$ is $r$ if and only if there is a non-trivial word $w$ of length $r$ in the letters $x,y$ which is trivial in $G$, and every non-trivial word of length smaller than $r$ is non-trivial in $G$. The girth of $\Gamma$ measures how similar the Cayley graphs $\Gamma$ and $\Gamma(F_2,\{X,Y\})$ are locally, where $F_2=F(X,Y)$ denotes the free group freely generated by $X$ and $Y$. Namely, if $\g(\Gamma)\ge 2r+2$, then the balls $B_r(\Gamma), B_r(\Gamma(F_2,\{X,Y\}))$ of radius $r$ with center in $e$ in these two graphs (i.e. the subgraphs induced by the vertices whose distance to $e$ is at most $r$) are isomorphic. 

\begin{teo} \label{main}
Let $G$ be a non-necessarily finite group generated by two elements $x,y$. Suppose that for every non-trivial word $w\in F_2$ with length smaller than or equal to $21$, $w(x,y)\neq e\in G$, that is $\g(\Gamma(G,\{x,y\}))>21$. Then $G$ admits a Cayley representation. 
\end{teo}
\begin{proof}
Let $S=\{e,x,x^2,x^4,y,y^3\}$ and let $P=P(G,S)$. Let $\varphi$ be an automorphism of $P$ fixing the minimal point $e\in P$. We want to prove that $\varphi=1_P$. Note that $N_S(e)=SS^{-1}$ consists of the elements $x^{-4},x^{-3},x^{-2},x^{-1},e,x,x^2,x^3,$ $x^4,y^{-3},y^{-2},$ $y^{-1},y,y^2,$ $y^3,xy^{-1},$ $xy^{-3},$ $x^2y^{-1},x^2y^{-3},x^4y^{-1},$ $x^4y^{-3}, yx^{-1},yx^{-2},yx^{-4}, y^3x^{-1},y^3x^{-2},y^3x^{-4}$. These 27 elements are different since the girth $\g(\Gamma)$ of $\Gamma=\Gamma(G,\{x,y\})$ is greater than 14. The ball $B_7(\Gamma)$ with center $e$ and radius $7$ in $\Gamma$ is isomorphic to the ball $B_7(\Gamma(F_2,\{X,Y\}))$ since $\g(\Gamma)\ge 16$. This ball is then a tree. In Figure \ref{cay} we have depicted the smallest connected subgraph of $B_7(\Gamma)$ which contains all the 27 vertices corresponding to elements in $N_S(e)$.

\begin{figure}[h] 
\begin{center}
\includegraphics[scale=0.45]{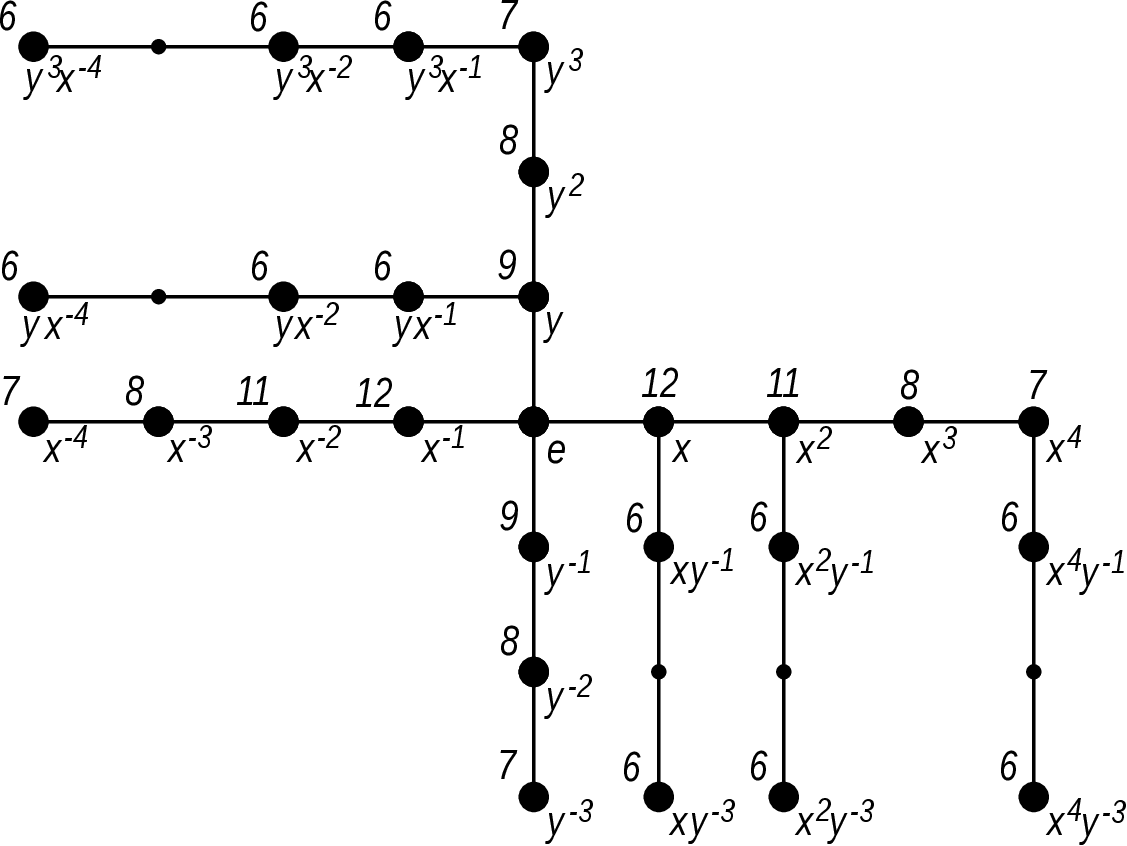}
\caption{The smallest connected subgraph of $B_7(\Gamma)$ containing the vertices of $N_S(e)$.}\label{cay}
\end{center}
\end{figure}

The 27 vertices corresponding to the points in $N_S(e)$ are represented with big dots, while there are 5 vertices in the picture which are not in $N_S(e)$, represented with small dots. In each vertex $e\neq g\in N_S(e)$ we have indicated the affinity $\alpha(e,g)=\#(N_S(e)\cap N_S(g))$. This can be computed in each of the 26 cases algebraically, using that no non-trivial word of length smaller than or equal to $21$ is trivial in $G$, or graphically, using that $\g(\Gamma)>21$. For the first alternative, note that an element $h\in N_S(e)$ is represented by a word $w\in F_2$ with total exponent smaller than or equal to 7, while an element $h'\in N_S(g)$ is represented by a word $w'$ of exponent at most 7+7=14. Thus $h=h'\in G$ if and only if $w=w'\in F_2$. For the second, graphical, alternative, note that if $g$ is represented by a word $u\in F_2$ of exponent at most $7$, then the union of $B_7(\Gamma(F_2, \{X,Y\}))$ and the ball $uB_7(\Gamma(F_2, \{X,Y\}))$ of center $u$ is a subgraph of $\Gamma(F_2,\{X,Y\})$ of diameter at most 21. Thus, this graph is isomorphic to the (non-necessarily induced) subgraph $B_7(\Gamma)\cup gB_7(\Gamma)$ of $\Gamma$, so $\alpha(e,g)$ can be computed in $\Gamma(F_2,\{X,Y\})$.  

Since $\varphi(e)=e$, $\varphi(N_S(e))=N_S(e)$. Among the points in $N_S(e)$, there are only two, $x^{-1},x$ whose affinity with $e$ is $\alpha(e,x^{-1})=\alpha(e,x)=12$. Thus $\varphi(x)\in \{x^{-1},x\}$. If $\varphi(x)=x^{-1}$, then $\varphi(N_S(x))=N_S(x^{-1})$ and $\alpha(\varphi(x^2),x^{-1})=\alpha(x^2,x)=12$. This implies that $\varphi(x^2)\in \{x^{-2},e\}$. Since $\varphi$ is injective, $\varphi(x^2)=x^{-2}$. By induction $\varphi(x^n)=x^{-n}$ for every $n\ge 0$. The set $\{e,x^2,x^3,x^4\}$ of minimal points of $P$ has an upper bound. However, $\varphi(\{e,x^2,x^3,x^4\})=\{e,x^{-2},x^{-3},x^{-4}\}$ does not. Indeed, if $g'\in P$ was an upper bound, then $\{e,x^{-2},x^{-3},x^{-4}\}\subseteq P_{\le g'}=\{g,gx^{-1},gx^{-2},gx^{-4},gy^{-1},gy^{-3}\}$. In particular $e$ must be contained in the later, so $g\in \{e,x,x^2,x^4,y,y^3\}=S$. But for any of these six cases we see that either $x^{-2}$ or $x^{-3}$ is not contained in $P_{\le g'}$, using that $\g(\Gamma)\ge 10$. We conclude then that $\varphi(x)=x$, and by an inductive argument, that $\varphi(x^n)=x^n$ for every $n\ge 0$. Similarly, $\varphi(x^{-n})=x^{-n}$ for every $n\ge 0$ (when $G$ is infinite, and moreover, $x$ is not of finite order, this does not follow from the previous claim). Moreover, this argument shows that if $g\in P$ is fixed by $\varphi$ for some $g\in G$, then $gx^n\in P$ is fixed for every $n\in \Z$. Among the points in $N_S(e)$ (we only care about points in $N_S(e)$ which are not powers of $x$, in fact), only $y^{-1}$ and $y$ have affinity equal to $9$ with $e$. Thus $\varphi (y)\in \{y^{-1},y\}$. As above, if $\varphi(y)=y^{-1}$, then $\varphi(y^n)=y^{-n}$ for every $n\ge 0$, and this yields a contradiction since $\{e,y^2,y^3\}\subseteq P$ has an upper bound, while $\{e,y^{-2},y^{-3}\}$ does not. Indeed, either $y^{-2}$ or $y^{-3}$ is not contained in $P_{\le g'}$ for each $g\in S$, as $\g(\Gamma)\ge 10$. Therefore, $\varphi(y)=y$, and by induction $\varphi(y^n)=y^n$ for every $n\ge 0$. The same holds for $n\le 0$. Moreover, for every $g\in G$, if $\varphi(g)=g$, then $\varphi(gy^n)=gy^n$ for every $n\in \Z$. Now, since $\{x,y\}$ generates $G$, all the minimal points of $P$ are fixed by $\varphi$. Finally, every maximal point of $P$ is determined by the set of points it covers. Concretely, $e'$ is the unique upper bound of $\{e,y^{-1}\}\subseteq P$, since for any $e\neq g\in S$, $y^{-1}$ does not belong to $P_{\le g'}$, using that $\g(\Gamma)\ge 9$. Moreover, for each $g\in G$, $g'$ is the unique upper bound of $\{g,gy^{-1}\}$. Since the later is invariant, $g'$ is also fixed. 
\end{proof}

\begin{obs}
Note than in the proof of Theorem \ref{main} we do not really need all the non-trivial words $w(X,Y)$ of length at most 21 to represent non-trivial elements of $G$, but just a concrete list with much fewer words. This could be used to obtain representations of examples not covered by the theorem. Also, even if some of the words in this list are trivial in $G$, a similar proof could work as long as we have some control on the number of words wich are indeed trivial. For example, the fact that $\Z^2$ has a Cayley representation follows from a variant of this type.
\end{obs}

Generalizations of Theorem \ref{main} for groups with $d\ge 3$ generators can be obtained with similar methods by combining $d$ generator sets $S,S',\ldots ,S^{(d)}$ of $\Z$ as we did for $d=2$. In this article we restrict ourselves to the case of two generators.

As a first example, we can apply Theorem \ref{main} to $G=F_2$.

The construction of regular graphs, and specifically Cayley graphs, with large girth, has a long history, with application to low density codes, among others. A simple counting argument due to Moore implies that a regular graph $\Gamma$ of degree 4 and girth $\g(\Gamma)>21$ has at least $3^{11}-1$ vertices. Thus, Theorem \ref{main} can only be applied to groups of order at least that number.   

In \cite{Mar} Margulis proves that if $p$ is a prime and $G=SL_2(\Z_p)$, then $x=\binom{1 \ 2}{0 \ 1}$, $y=\binom{1 \ 0}{2 \ 1}$ generate $G$ and $\g(\Gamma(G, \{x,y\}))\ge 2 \log_{1+\sqrt{2}}(\frac{p}{2})-1$. Thus, for $p>2(1+\sqrt{2})^{11}$, $SL_2(\Z_p)$ admits a Cayley representation. For $PGL_2(\Z_p)$ and $PSL_2(\Z_p)$ there are similar results by Lubotzky, Phillips and Sarnak \cite{LPS}, but their Cayley graphs are constructed with $\frac{q+1}{2}$ generators, where $q$ is a prime congruent to $1$ modulo $4$.

%

\section{Representability of finite simple groups}

In \cite{Wei} Weigel proved the following result, answering a question originally raised by Magnus:

\begin{teo}(Weigel) \label{weigel}
If $\mathcal{C}$ is an infinite family of isomorphic types of non-abelian finite simple groups, then $F_2=F(X,Y)$ is residually $\mathcal{C}$. That is, the intersection of all the normal subgroups $N\trianglelefteq F_2$ such that $F_2/N\in \mathcal{C}$, is trivial.
\end{teo}

In fact, this result holds for $F_d$, $d\ge 2$. A much stronger result is proved by Dixon, Pyber, Seress and Shalev in \cite{DPSS} using probabilistic methods. Although our result in this section (Corollary \ref{corosimple}) follows from  \cite[Theorem 3]{DPSS} (and Theorem \ref{main}), we choose an approach using only Weigel's result. As explined in \cite{DPSS}, Weigel's Theorem implies 

\begin{teo} \label{weigel2}
If $\mathcal{S}$ is an infinite family of non-isomorphic non-abelian finite simple groups and $w\in F_2=F(X,Y)$ is a non-trivial word, then there exists $S\in \mathcal{S}$ and $x,y\in S$ such that $x,y$ generate $S$ and $w(x,y)\neq 1\in S$.
\end{teo}

Indeed, since $w$ is non-trivial, by Theorem \ref{weigel} there exists $N\trianglelefteq F_2$ such that $w\notin N$ and $F_2/N$ is isomorphic to some $S\in \mathcal{S}$. Let $\varphi :F_2 \to S$ be an epimorphism with ker$(\varphi)=N$. Then $x=\varphi(X)$, $y=\varphi(Y)$ satisfy the required property.

If $P_1,P_2,\ldots,P_k$ are non-trivial polynomials over some coefficient ring, then their product $P$ has the property that every root of some $P_i$ is also a root of $P$. The following is a similar construction for equations over groups. Recall that an equation in $d$ variables is just a word $w\in F_d=F(X_1,X_2,\ldots , X_d)$, and a solution of $w$ on a group $G$ is a $d$-tuple $(x_1,x_2,\ldots,x_d)\in G^d$ such that $w(x_1,x_2,\ldots,x_d)=1\in G$. 

\begin{lema} \label{polinomios}
Let $w_1,w_2,\ldots, w_k \in F_d$ be non-trivial words. Then there exists a non-trivial word $w\in F_d$ with the following property. For every group $G$ and for every $1\le i\le k$, every solution of the equation $w_i$ on $G$, is also a solution of $w$.
\end{lema} 
\begin{proof}
By induction it suffices to prove the result for $k=2$. Let $F$ be the subgroup of $F_d$ generated by $w_1,w_2$. If rk$(F)=1$, $F$ is generated by some non-trivial word $u\in F_d$ and there exist $l,m\in \Z$ with $w_1=u^l$, $w_2=u^m$. By assumption $l,m\neq 0$. Let $w=u^{lm}$. Then $w$ satisfies the required property. If rk$(F)=2$, then $\{w_1,w_2\}$ is a basis of $F$ and, in particular, $w_1$ and $w_2$ do not commute. Take $w=[w_1,w_2]\in F_d$. Then $w$ is non-trivial, and every solution of $w_1$ and of $w_2$ on a group $G$ is a solution of $w$.
\end{proof}

\begin{coro} \label{corosimple}
There are only finitely many finite simple groups which do not admit a Cayley representation. 
\end{coro}
\begin{proof}
By Example \ref{ciclico} there are only finitely many finite abelian simple groups which do not admit such representation. Suppose there exists an infinite family $\mathcal{S}$ of non-isomorphic non-abelian finite simple groups which do not admit such representation. Let $w_1,w_2,\ldots, w_k$ be all the non-trivial words in $F_2$ of length smaller than or equal to $21$. Let $w\in F_2$ be a non-trivial word as in the statement of Lemma \ref{polinomios}. By Theorem \ref{weigel2} there exists $S\in \mathcal{S}$ and $x,y\in S$ which generate $S$ and such that $w(x,y)\neq 1\in S$. Then $w_i(x,y)\neq 1\in S$ for every $1\le i\le k$. This is a contradiction by Theorem \ref{main}.
\end{proof}


\section{Dehn presentations and random groups}

In this section we will find a large class of examples satisfying the hypothesis of Theorem \ref{main}, and we will see that, in some sense, almost all finitely presented groups admit a Cayley representation.

We assume the reader is familiar with small cancellation theory. Standard reference for this is \cite{LS, Ol}. If a presentation $\mathcal{P}=\langle X | R\rangle$ satisfies the $C'(\frac{1}{6})$ condition, then by Greendlinger's lemma, it is a Dehn presentation (i.e. the Dehn algorithm solves the word problem). In particular, if the length of every relator in $\mathcal{P}$ is greater than or equal to some number $l\ge 1$, then $\g(\Gamma(G_\mathcal{P},X))\ge l$.  

\begin{coro} \label{sixth}
Let $\mathcal{P}=\langle x,y | r_1,r_2,\ldots, r_m \rangle$ be a presentation which satisfies $C'(\frac{1}{6})$ and such that length$(r_j)\ge 22$ for every $1\le j\le m$. Then the presented group $G_{\mathcal{P}}$ admits a Cayley representation.
\end{coro}

For 1-relator groups with torsion there is a stronger result than Corollary \ref{sixth} in virtue of the Newman Spelling Theorem, which implies that when the relator is a proper power, the presentation is Dehn.

\begin{coro}
Let $\mathcal{P}=\langle x,y| r^m\rangle$ be a presentation, were $r\in F_2$ is cyclically reduced and $m\ge 2$. If the length of $r^m$ is greater than or equal to $22$, $G_\mathcal{P}$ admits a Cayley representation.
\end{coro}

When studying random groups, there are two models which have received more attention than any other: Gromov's density model introduced in \cite[\S 9]{Gro} and Arzhantseva and Ol'shanskii's few relators model \cite{AO}. We recall basic definitions. Standard reference on this is \cite{Ghy, KS, Oll}. In the density model we fix a number $n\ge 2$ of generators $x_1,x_2,\ldots, x_n$ and a \textit{density parameter} $0\le d\le 1$. The number of cyclically reduced words of length $l$ is asymptotically $(2n-1)^l$. Given $l\ge 1$, we choose $(2n-1)^{dl}$ cyclically reduced words of length $l$ uniformly randomly and independently to get a presentation $\mathcal{P}=\langle x_1,x_2,\ldots, x_n | R \rangle$ whose relators are those words. We say that a random presentation (or a random group) satisfies a certain property, if the probability that $\mathcal{P}$ (or $G_\mathcal{P}$) satisfies the property tends to $1$ as $l\to \infty$. We have the following result by Gromov and Ollivier \cite{Oll}.

\begin{teo} (Gromov, Ollivier)
In the density model, if $d<\frac{1}{2}$, a random group is infinite hyperbolic, while for $d>\frac{1}{2}$, a random group is $\Z$ or $\Z_2$.
\end{teo}

Although hyperbolic groups admit a Dehn presentation, this new presentation could have relators of different lengths. However the following result holds.

\begin{teo} (Gromov \cite{Gro})
In the density model, if $d<\frac{1}{12}$, then a random presentation satisfies $C'(\frac{1}{6})$, so the Dehn algorithm solves the word problem. If $d>\frac{1}{12}$, a random presentation does not satisfy $C'(\frac{1}{6})$.
\end{teo}

The fact that a presentation does not satisfy $C'(\frac{1}{6})$ does not imply it is not a Dehn presentation. Moreover, Ollivier proved in \cite{Oll2} the following result.

\begin{teo} (Ollivier)
In the density model, if $d<\frac{1}{5}$, a random presentation is Dehn's. If $d>\frac{1}{5}$, then it is not.
\end{teo}

This, together with Theorem \ref{main} imply the following

\begin{coro} \label{corodensity}
In the density model, if $d<\frac{1}{5}$, a random group with $n=2$ generators admits a Cayley representation.
\end{coro}

Similar results for every fixed number $n\ge 3$ of generators should be true, and a generalization of Theorem \ref{main} could be used in the proof.

We turn now to the few relators model by Arzhantseva and Ol'shanskii. In this model both the number $n$ of generators and number $m$ of relators are fixed. For each $l\ge 1$, $m$ cyclically reduced words of length at most $l$ are chosen at random independently and uniformly to form the set $R$ of relators. We say that a random $n$-generator, $m$-relator presentation (or group) satisfies certain property if it does with probability $\to 1$ when $l\to \infty$.       

\begin{teo} (Arzhantseva, Ol'shanskii \cite{AO}) \label{AOteo}
For every $n\ge 2$, $m\ge 1$, $\lambda>0$, a random $n$-generator $m$-relator presentation satisfies $C'(\lambda)$.
\end{teo}

\begin{coro} \label{corofew}
For every $m\ge 1$, a random $2$-generator $m$-relator group admits a Cayley representation. 
\end{coro}
\begin{proof}
This follows from Corollary \ref{sixth}, by applying Theorem \ref{AOteo} for $n=2, \lambda=\frac{1}{6}$ and this assertion: in a random $2$-generator $m$-relator presentation all the relators have length at least $22$. This follows from a simple counting argument (cf. \cite[p. 3208]{Arz}): the number $c_l$ of cyclically reduced words in $X,Y$ of length $l$ is smaller than or equal to $4.3^{l-1}$. Thus, the number $c_{\le l}$ of cyclically reduced words of length $\le l$ is at most $2.(3^l-1)<2.3^l$. The number of $m$-tuples of cyclically reduced words of length at most $l$ which contain at least one word of length $\le 21$ is at most $m.2.3^{21}(2.3^{l})^{m-1}=m2^m3^{21}3^{lm-l}$. On the other hand $c_{\le l}\ge c_l\ge 4.3^{l-2}.2=8.3^{l-2}$ for each $l\ge 2$. Therefore, the number of $m$-tuples of cyclically reduced words of length $\le l$ is at least $(8.3^{l-2})^m=8^m3^{lm-2m}$. Since $m2^m3^{21}3^{lm-l}/8^m3^{lm-2m}\to 0$ as $l\to \infty$, the assertion is proved.  
\end{proof}

\section{Extensions of the integers by semi-regularly representented groups} \label{sectionregular}

Recall that a group $G$ is said to be indicable if there exists an epimorphism $G\to \Z$. In other words, they are the extensions of $\Z$ or, equivalently, the semidirect products $N\rtimes_\psi \Z$.

\begin{teo} \label{producto1}
Let $G\neq \Z_2$ be a group which admits a Cayley representation, and let $\psi:\Z\to \aut (G)$ be a group homomorphism. Then $G\rtimes_\psi \Z$ admits a regular poset representation. 
\end{teo}
\begin{proof}
Let $P(G,S)$ be a Cayley representation of $G$. Define $P$ to be the poset obtained from countable many copies $P(G,S)\times \{n\}$ ($n\in \Z$) of $P(G,S)$ with the identifications $(g',n)\sim (\psi(1)(g),n+1)$ for every $g\in G, n\in \Z$. That is, the order on $P$ is the transitive closure of the union of the orders in each copy. 
Let $H=(G^{op}\rtimes_\psi \Z)^{op}$. In other words, the underlying set of $H$ is the cartesian product $G\times \Z$ and the operation is defined by $(g_1,n_1)(g_2,n_2)=(\psi(n_2)(g_1)g_2, n_1+n_2)$. The groups $G\rtimes_\psi \Z$ and $H$ are isomorphic via the map $(g,n)\mapsto (\psi(-n)(g),-n)$. Note that $H$ has a well-defined left action on $P$ given by $(g_1,n_1)\cdot (g_2,n_2)= (\psi(n_2)(g_1)g_2, n_1+n_2)$ and $(g_1,n_1)\cdot (g_2',n_2)= ((\psi(n_2)(g_1)g_2)', n_1+n_2)$. The action is clearly transitive and free. In order to show that $H\to \aut(P)$ is surjective it suffices to prove that any automorphism of $P$ fixing $(e,0)$ is the identity.

Suppose $\varphi \in \aut(P)$ fixes $(e,0)$. Since $P(G,S)$ is connected by Remark \ref{generan}, $\varphi$ restricts to an automorphism of $P(G,S)\times \{0\}$. Concretely, the set $D_1$ of points covering $(e,0)$ is $\varphi$-invariant, the set $D_2$ of points covered by those in $D_1$ is then also is $\varphi$-invariant, then the set $D_3$ of points covering those in $D_2$, and so on. Thus $P(G,S)\times \{0\} \subseteq P$ is $\varphi$-invariant. Since it is also $\varphi^{-1}$-invariant, $\varphi(P(G,S)\times \{0\})=P(G,S)\times \{0\}$. 

By hypothesis $\varphi|_{P(G,S)\times \{0\}}$ is the identity. In particular $(e,1)=(e',0)$ is fixed. By an inductive argument $\varphi|_{P(G,S)\times \{n\}}$ is the identity for every $n\ge 0$. Also, since $(e',-1)=(e,0)$ is fixed, then $P(G,S)\times \{-1\}$ is invariant, and $\varphi |_{P(G,S)\times \{-1\}}$ is the identity. By induction we deduce that $\varphi$ is the identity of $P$.    
\end{proof}

\begin{ej}
Let $G$ be the fundamental group of the Klein bottle, that is the group presented by $\langle x,y| \ xyx^{-1}y\rangle$. Then $G$ admits a regular poset representation. Indeed, there is a split short exact sequence $$1\to \Z \to G\to \Z\to 1,$$ where the map $\Z\to G$ takes a generator of $\Z$ to $y$, and $G\to \Z$ maps a word $w$ to its total $x$-exponent. Thus $G=\Z \rtimes_\psi \Z$ for some homomorphism $\psi:\Z \to \aut(\Z)$. Since $\Z$ admits a Cayley representation $P(\Z,S)$ for $S=\{0,1,3\}$ (see the comments after Example \ref{ciclico}), Theorem \ref{producto1} applies.
\end{ej}

In view of the results in previous sections, many groups with torsion (non-trivial elements of finite order) admit a regular poset representation. On the other hand if $G\neq 1, \Z_2$ is a torsion group, then it admits no regular representation. If $P$ is a non-discrete regular representation of a torsion group $G$, there are comparable elements $x<y$, so there exists $g\in G$ with $gx=y$. This implies that $g^nx<g^{n+1}x$ for every $n\ge 0$, and the fact that $g$ has finite order yields a contradiction. If $P$ is a discrete regular representation of $G$, then $G$ is trivial or of order $2$. For example, the Burnside groups $B(d,n)$ of exponent $n$ and $d\ge 2$ generators are finitely generated groups which do not admit a regular representation, and they are infinite for $n$ large enough \cite{Adi, Iva}. 

The posets $P$ constructed in Theorem \ref{producto1} are regular representations of indicable groups which are \textit{graded} in the sense that there exists a (\textit{rank}) function $\rho:P\to \Z$ such that 1. $x<y$ implies $\rho (x)<\rho (y)$ and 2. whenever $y$ covers $x$, we have $\rho(y)=\rho(x)+1$. Not every regular representation of a group is graded. Let $P$ be the poset with underlying set $\Z$ and the order $\vartriangleleft$ given by $a \vartriangleleft b$ if $a=b$ or $b-a\ge 2$. Then it is clear that $P$ is a regular representation of $G=\Z$. Indeed, if an automorphism $\varphi$ of $P$ fixes $0$, the set $\{-1,1\}$ of points not comparable with $0$ must be $\varphi$-invariant. Since $-1\vartriangleleft 1$, $\varphi$ fixes both $-1$ and $1$, and by an inductive argument $\varphi=1_P$. On the other hand $P$ is not graded as $0\vartriangleleft 2\vartriangleleft 4 \vartriangleleft 6$ and $0 \vartriangleleft 3 \vartriangleleft 6$ are two maximal chains from $0$ to $6$ of different length.

\begin{prop}
Let $G\neq 1, \Z_2$ be a group which admits a graded regular poset representation. Then $G$ is indicable. 
\end{prop}
\begin{proof}
Let $P$ be a graded regular representation of $G$ with rank function $\rho:P\to \Z$. Since $G\neq 1, \Z_2$, by Remark \ref{generan}, $P$ must be connected. Fix $x\in P$. By adding a constant to $\rho$ we may assume $\rho(x)=0$. Define $f:G\to \Z$ by $f(g)=\rho(gx)$. We prove that $f$ is a group homomorphism. Let $g,h\in G$. Condition 1 in the definition of rank function guarantees that there are no infinite chains between two fixed elements of $P$. By connectivity there is a sequence $x=x_0,x_1,\ldots ,x_n=gx$ in which for every $0\le i< n$ either $x_{i+1}$ covers $x_{i}$ or $x_i$ covers $x_{i+1}$. Then, the number of $i$ for which the first happens minus the number of $i$ for which the latter happens is $\rho(gx)$. Similarly, there is a sequence $x=y_0,y_1,\ldots , y_m=hx$ in which each $y_i$ is covered by or covers $y_{i+1}$, and the corresponding difference is $\rho(hx)$. Thus, $$x_0, x_1, \ldots ,x_n, gy_1,gy_2, \ldots ghx$$ is a sequence of the same kind, and the corresponding difference is $\rho(gx)+\rho(hx)=\rho(ghx)$. Thus $f(gh)=f(g)+f(h)$. The fact that $f$ is an epimorphism follows directly from the fact that $P$ is non-discrete. 
\end{proof}

\begin{teo} \label{producto2}
Let $G$ be a group different from $\Z_2^2,\Z_2^3,\Z_2^4,\Z_3^2$ and let $\psi :\Z \to \aut(G)$ be a group homomorphism. Then $G\rtimes_\psi \Z$ admits a semi-regular poset representation with two orbits.
\end{teo}
\begin{proof}
The proof is very similar to that of Theorem \ref{producto1}. If $G$ is the trivial group, the existence of a non-Cayley and a Cayley representation was already discussed in the beginning of Section \ref{sectioncayley} and right after Example \ref{ciclico}. Assume $G$ is non-trivial. The representation we construct will not be Cayley. If $G\neq Q_8$, let $B$ be a semi-regular representation of $G$ with three orbits, as constructed by Babai in \cite{Bab2, Bab} and recalled after Theorem \ref{DRR} above. The underlying set of $B$ is $G\cup G'\cup G''$. The set of minimal points of $B$ is $G$, and the set of maximal points is $G''$. A point $g'\in G'$ covers just one element $g\in G$ and it is covered only by $g''\in G''$. If $G\neq \Z_2$, $B$ is necessarily connected, and if $G=\Z_2$, $B$ can and will be assumed to be connected. In particular, since $G$ is non-trivial, every minimal point $g$ is covered by at least one maximal point $h''$ (and also by $g'$).

Consider countable many copies $B\times \{n\}$ ($n\in \Z$) of $B$. Identify $(g'',n)$ with $(\psi(1)(g),n+1)$ for every $g\in G, n\in \Z$, to obtain a poset $P$. Note that the set $B'$ of points $(g',n)\in P$ for $g\in G, n\in \Z$, is invariant by any automorphism of $P$ (and so is its complement), since those are the points covered by just one element.

Let $H$ be the group isomorphic to $G\rtimes_\psi \Z$ defined in the proof of Theorem \ref{producto1}. There is a left action $\lambda$ of $H$ on $P$ given by $(g_1,n_1)\cdot (g_2,n_2)= (\psi(n_2)(g_1)g_2, n_1+n_2)$, $(g_1,n_1)\cdot (g_2',n_2)= ((\psi(n_2)(g_1)g_2)', n_1+n_2)$, $(g_1,n_1)\cdot (g_2'',n_2)= ((\psi(n_2)(g_1)g_2)'', n_1+n_2)$. This action is free with two orbits: $B'$ and its complement. Since these are invariant by any automorphism, to show that $\lambda$ is the required representation we may prove that every automorphism of $P$ fixing $(e,0)\in G\times \{0\}$ is the identity.   


Let $\varphi\in \aut(P)$ be such that $\varphi(e,0)=(e,0)$. Suppose that $D$ is a subset of $B\times \{0\}$ which is invariant in this strong sense: $D\cap (G\times \{0\})$, $D\cap (G'\times \{0\})$ and $D\cap (G''\times \{0\})$ are $\varphi$-invariant. We claim that the set $\overline{D}$ of points in $B\times \{0\}$ which are greater than some point in $D$ also satisfies this property. Indeed, if $(g',0)$ covers some element in $D\cap (G\times \{0\})$, then so does $\varphi(g',0)$. This implies that $\varphi(g',0)\in B\times \{0\}$, and since $B'$ is invariant, $\varphi(g',0)\in G'\times \{0\}$. Now, if $(h'',0)$ covers an element in $D\cap (G'\times \{0\})$, or if it covers an element in $D\cap (G\times \{0\})$, or if it covers an element $(g',0)$ which covers a point in $D\cap (G\times \{0\})$ (as before), then so does $\varphi(h'',0)$. In particular $\varphi(h'',0)\in B\times \{0\}$. Since the complement of $B'$ is $\varphi$-invariant, $\varphi(h'',0)\in G''\times \{0\}$. Similarly, the set $\underline{D}$ of points in $B\times \{0\}$ which are smaller than some point in $D$ is also invariant in the strong sense. Let $D_0=\{(e,0)\}$. For $k\ge 1$ odd define $D_k=\overline{D_{k-1}}$, while for $k\ge 2$ even, define $D_k=\underline{D_{k-1}}$. Then $D_k$ is $\varphi$-invariant for every $k\ge 0$. Since $B$ is connected, $B\times \{0\}=\cup D_k$ is $\varphi$-invariant. Since it is also $\varphi^{-1}$-invariant, $\varphi$ restricts to an automorphism of $B\times \{0\}$ fixing a point. Therefore this is the identity. By induction $\varphi$ restricts to the identity in each copy of $B$, as we wanted to prove.

The case $G=Q_8$ is very similar to the previous. In this case, the semi-regular representation $B$ of $G$ with three orbits constructed by Babai in \cite{Bab} (proof of Corollary 4.3) also has $G\cup G'\cup G''$ as underlying set, with the set of minimal points being $G$ and maximal points being $G''$. Every point in $G$ is covered by three points: two in $G'$ and one in $G''$, while each point in $G'$ is covered by two points. With this, essentially the same proof works as the set $B'$ of points $(g',n)\in P$ is invariant by any automorphism.
\end{proof}

Since for each exception $G=\Z_2^2,\Z_2^3,\Z_2^4,\Z_3^2$ there are only finitely many homomorphisms $\psi:\Z\to \aut(G)$, we deduce the following

\begin{coro}
With finitely many exceptions, every indicable group admits a semi-regular representation with two orbits.
\end{coro}

In particular every torsion-free indicable group admits a semi-regular representation with two orbits, and thus every locally indicable group (that is, each non-trivial finitely generated subgroup is indicable) admits such a representation. This includes knot groups \cite{How}, torsion free one-relator groups \cite{Bro, How}, amenable left-orderable groups \cite{Mor}.
%

\bigskip

We finish with the postponed proof that $\Z_2^2$ does not admit a semi-regular representation with three orbits.

\bigskip

\textit{Proof of Proposition \ref{zeta22}}. Assume $P$ is a semi-regular poset representation of $G=\Z_2^2=\{0,a,b,a+b\}$ with three orbits. We can identify those orbits with three copies $G,G',G''$ of $G$, and  assume the action of $G$ is the regular left action on each orbit. Since points in the same orbit are not comparable, the height $h(P)$ of $P$ is 1 or 2.

Case 2: $h(P)=2$. We may assume by a relabeling that $0<0'<0''$. If $0'$ is only comparable with $0$ and $0''$, then $P$ is the poset associated to a digraphical regular representation of $G$ (see the proof of Theorem 7.3 in \cite{Bab2}), which is absurd by Theorem \ref{DRR}. Since the opposite $P^{op}$ of $P$ is also a semi-regular representation, we may assume $0'<a''$.

We call an edge $(x,y)$ in the Hasse diagram $\mathcal{H}(P)$ of $P$ \textit{long} if $x\in G$ and $y\in G''$. If there are two long edges starting at $0$, they must be $(0,b'')$ and $(0,(a+b)'')$. In this case we claim that the poset $\widetilde{P}$ obtained from $P$ by removing from the relation all the pairs $(x,y)$ such that $(x,y)$ is a long edge, has the same automorphism group as $P$. Indeed, $x\in G$ and $y\in G''$ are not comparable in $\widetilde{P}$ if and only if $(x,y)$ is a long edge of $\mathcal{H}(P)$. Thus we may assume there is at most one long edge starting at $0$.

Case 2.0: There is no long edge starting at $0$. In other words, there are no long edges at all in $\mathcal{H}(P)$. If $0'$ is only covered by $0''$ and $a''$ or if it is covered by all the elements in $G''$, the transposition which permutes $0''$ and $a''$ shows that the action of $G$ on $P$ is not semi-regular. We may assume that $0'$ is covered by $0'',a'',b''$ and not by $(a+b)''$. In this case $\aut(P)$ is isomorphic to the automorphism group of the subposet induced by $G\cup G'$, so $G$ admits a semi-regular representation with two orbits, a contradiction by Proposition \ref{contraejemplos}. 

Case 2.1: There is one long edge starting in $0$, which may be assumed to be $(0,b'')$. Since $b''$ covers $0$, $0'\nless b''$ and $0\nless b'$, so $b=b\cdot 0 \nless b\cdot b'= 0'$. Since $0'<a''$, $(a+b)'=(a+b)\cdot 0'<(a+b)\cdot a''=b''$, and since $b''$ covers $0$, $0\nless (a+b)'$. Then $a+b\nless 0'$. Thus $0'$ covers $0$ and possibly $a$, but nothing else. And $0'$ is covered by $0''$ and $a''$, and possibly $(a+b)''$, but not $b''$. Thus, there are only four cases to analyze. However, we cannot have $a<0'$ and $0'<(a+b)''$ simultaneously as $(a,(a+b)'')$ is a long edge. The three possible cases appear in Figure \ref{tres}.

\begin{figure}[h] 
\begin{center}
\includegraphics[scale=0.36]{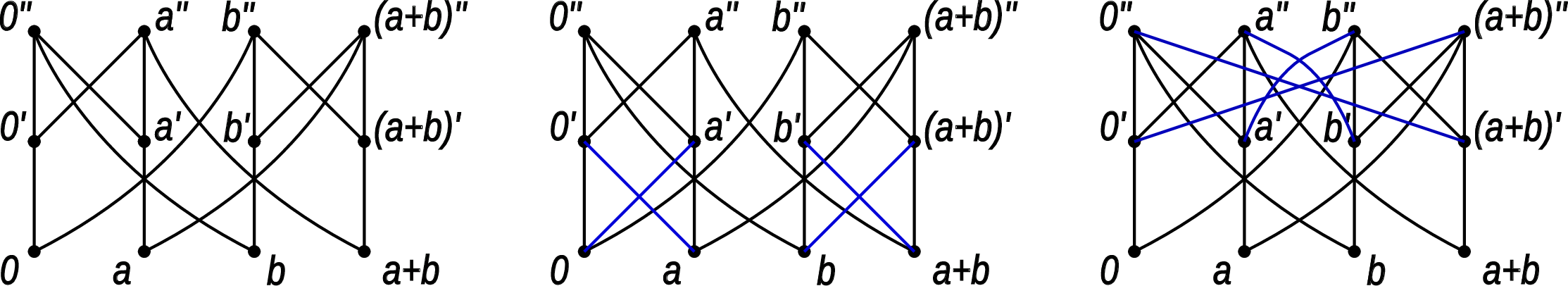}
\caption{Three candidates for Case 2.1.}\label{tres}
\end{center}
\end{figure}

In the first and second cases we have the involution which permutes $0''$ with $a''$, $b'$ with $(a+b)'$ and $b$ with $a+b$. In the third case there is an involution permuting $a''$ with $(a+b)''$, $a'$ with $(a+b)'$, and $a$ with $a+b$. These are non-identity automorphisms fixing $0$, a contradiction.


Case 1: $h(P)=1$. Since $P$ is connected by Remark \ref{generan}, by a relabeling we may assume that $0'>0<0''$. The dual $0'<0>0''$ can be ignored as $P^{op}$ is another semi-regular representation.

Since $h(P)=1$, $0'$ covers only points from the orbit $G$. If $0'$ covers four points, the transposition which permutes $0'$ and $a'$ shows that the action of $G$ on $P$ is not free. If $0'$ covers one point or three, $\aut(P)$ is isomorphic to the automorphism group of the subposet induced by $G\cup G''$. Thus, $G$ would admit a semi-regular representation with two orbits, a contradiction. We may assume that $0'$ covers two elements, $0$ and $a$. Then the transposition which permutes $0'$ and $a'$ is a non-trivial automorphism of $P$ which fixes $0$.
\hfill \qed  

\end{document}